\documentclass[11pt,reqno]{amsart}
\usepackage{amsmath}
\usepackage{cases}
\usepackage{mathrsfs}
\usepackage{bbm}
\usepackage{amssymb}
\usepackage{amscd}
\usepackage{amsfonts,latexsym,amsmath,
amsthm,amsxtra,mathdots,amssymb,latexsym,mathabx}
\usepackage[all,cmtip]{xy}
\RequirePackage{amsmath} \RequirePackage{amssymb}
\usepackage{color}
\usepackage{colordvi}
\usepackage{multicol}
\usepackage{hyperref}
\usepackage{mathtools}
\usepackage[margin=1in]{geometry}
\usepackage{xcolor}
\usepackage{comment}
\usepackage{cancel}

\hypersetup{
        colorlinks,
        linkcolor={red!50!black},
        citecolor={blue!100!black},
        urlcolor={blue!100!black}
}
\usepackage{cite}

\newcommand{\bea}{\begin{eqnarray}}
\newcommand{\eea}{\end{eqnarray}}
\newcommand{\bna}{\begin{eqnarray*}}
\newcommand{\ena}{\end{eqnarray*}}

\numberwithin{equation}{section}

\theoremstyle{plain}
\newtheorem{lemma}{Lemma}[section]

\newtheorem{theorem}[lemma]{Theorem}

\newtheorem{proposition}[lemma]{Proposition}

\theoremstyle{definition}

\newtheorem{remark}{Remark}
\newtheorem{conjecture}[lemma]{Conjecture}

\renewcommand{\Im}{\operatorname{Im}}

\makeatletter
\@namedef{subjclassname@2020}{%
  \textup{2020} Mathematics Subject Classification}

\title{Explicit bounds for the graphicality of the prime gap sequence}

\author{Keshav Aggarwal}
\address{{Department of Mathematics, Indian Institute of Technology Bombay, Mumbai, India}}
\email{{keshav@math.iitb.ac.in}}
\author{Robin Frot}
\address{DEDAUB Ltd., Malta Life Sciences Park, San Gwann, SGN 3000, Malta}
\email{{robin.frot@gmail.com}}
\author{Haozhe Gou}
\address{School of Mathematics, Shandong University, Jinan 250100, China}
\email{hodgegou@mail.sdu.edu.cn}
\author{Hui Wang}
\address{The University of Hong Kong Shenzhen Institute of Research and Innovation, 
Shenzhen, Guangdong 518052, China}
\email{wh0315@mail.sdu.edu.cn}


\begin{document}

\maketitle

\centerline{\emph{Dedicated to J\'anos Pintz on the occasion of his 75th birthday}}

\begin{abstract}
We establish explicit unconditional results on the graphic properties of the prime gap sequence.
Let \( p_n \) denote the \( n \)-th prime number (with $p_0=1$)
and \( \mathrm{PD}_n = (p_\ell - p_{\ell-1})_{\ell=1}^n \) be the sequence of the first \( n \) prime gaps. Building upon the recent work
by Erd\H{o}s \emph{et al}, %\cite{EHKMMT}
% (Math. Ann., \textbf{388} (2024), no. 2, pp. 2195--2215)
which proved the graphic nature of \( \mathrm{PD}_n \) for large \( n \) unconditionally, and for all \( n \) under RH, we provide the first explicit unconditional threshold such that:
\begin{enumerate}
\item For all \( n \geq \exp\exp(30.32) \), \( \mathrm{PD}_n \) is graphic.
\item For all \( n \geq \exp\exp(34.33) \), every realization \( G_n \) of \( \mathrm{PD}_n \)
satisfies that \( (G_n, p_{n+1}-p_n) \) is DPG-graphic.
\end{enumerate}

Our proofs utilize a more refined criterion for when a sequence is graphic, and better estimates for the first moment of large prime gaps proven through an explicit zero-free region and explicit zero-density estimate for the Riemann zeta function.  %Unlike prior conditional results, our bounds eliminate dependence on RH, though at the cost of double-exponential magnitude.
        \end{abstract}

\section{Introduction}\label{sec:introduction}
Let $p_n$ denote the $n$-th prime number with $p_0=1$.
We call the sequence $$\mathrm{PD}_n=\left(p_{\ell}-p_{\ell-1}\right)_{\ell=1}^{n}=(1,1,2,2,4,2,\cdots,{p_n-p_{n-1}})$$
        the \textit{first $n$ prime gap sequence}.
        A sequence of nonnegative integers \( (d_\ell)_{\ell=1}^{n}\) is said to be \textit{graphic}
        if there exists a simple undirected graph with $n$ vertices
        whose degrees are exactly \(d_1, d_2, \dots, d_n\). {We call a simple graph on $n\ge2$ vertices a \emph{prime gap graph} if its vertex degrees are $\mathrm{PD}_n$.}

        Two key conjectures have inspired this study:
        \begin{conjecture}[Existence; Toroczkai, 2016]\label{Conj1}
                For every ${n \geq2} $, there exists a prime gap graph on \( n \) vertices.
        \end{conjecture}
        \begin{conjecture}[Structural; Toroczkai, 2016]\label{Conj2}
                Every prime gap graph on \( n \) vertices contains \( (p_{n+1} - p_n)/2 \) independent edges.
        \end{conjecture}

        Notably,  Conjecture \ref{Conj2} implies Conjecture \ref{Conj1} (see the next subsection), linking the combinatorial properties of
        prime gap graphs to the distribution of primes.
        By merging tools from analytic number theory (to analyze prime gaps)
        and matching theory (to establish graph realizability and edge independence), Erd\H{o}s \emph{et al.} \cite{EHKMMT} proved these two conjectures hold for sufficiently large $n$ unconditionally, and for every $n$ assuming the Riemann Hypothesis (RH).

        This paper establishes an explicit lower bound \(n_0\) such that the result holds unconditionally for all \(n \geq n_0\).
        Our main theorem is:
        \begin{theorem}\label{MainThm}
               One may take $n_0= \exp\exp(30.32)$. That is, for all \(n \geq \exp\exp(30.32)\), the first \(n\) prime gap sequence $\mathrm{PD}_n$ is graphic.
        \end{theorem}

        While our results broadly follow the framework of~\cite{EHKMMT},
        achieving an optimal explicit bound \(n_0\) requires two critical refinements: a sharper criterion for graphicality of degree sequences (developed in Section \ref{sec:preliminaries}),
        and a refined treatment of zero-free regions and zero-density estimates for the Riemann zeta function (detailed in Section \ref{sec:first moment}).

        The main bottleneck is the known explicit zero-free region for the Riemann zeta function, which impacts the explicit bounds in prime gap estimates.
        
        \subsection{DPG-process}
        %\label{sec:introduction}
        The \textit{Degree-Preserving Network Growth (DPG)} process is an algorithm to construct a sequence of graphs that iteratively extends a graph \( G_n \) with degree sequence
        \( D_n = (d_1, \ldots, d_n) \) to \( G_{n+1} \) with \( D_{n+1} = D_n\,\circ\,d_{n+1}:= (d_1, \ldots, d_n, d_{n+1})\),
        while preserving the degrees of the vertices of \(G_n\).
        Unlike static algorithms like the Havel--Hakimi algorithm,
        which recompute realizations from scratch, DPG modifies \( G_n \) through a local operation:

        \textbf{Degree Preserving-Step}: Insert a new vertex \( u \) with even degree \( d_{n+1} = 2\nu \).
        \begin{enumerate}
                \item Remove a matching of size \( \nu \) from \( G_n \)
                (i.e., a set of $\nu $ independent edges with no shared vertices).
                \item Connect \( u \) to all \( 2\nu \) endpoints of the removed edges.
        \end{enumerate}
      Existing vertices retain their degrees, as each loses one edge from the removed matching but gains a connection to \( u \).

%\begin{theorem}[Theorem 2.5 in \cite{EMTZ24}]
 %   Given a graphic sequence $D$ of length n and an even integer $2 \le d \le n$, the sequence $D\circ d $ is graphic if and only if $D$ has a realization with a matching of size $d/2$.
%\end{theorem}

{ The pair \( (G_n, d_{n+1}) \) is called \textit{DPG-graphic} if  \( G_n \) contains a matching of size \(  d_{n+1}/2 \). 
 By \cite[Theorem 2.4]{EMTZ24}, the sequence $D_n \circ d_{n+1}$ is graphic if and only if there exists a realization $G_n$ of $D_n$ such that the pair $(G_n,d_{n+1})$ is DPG-graphic. In order to construct an infinite DPG process,  one needs the requirement that every realization of $D_n$ contains a matching of size $d_{n+1}/2$.}
While this holds trivially for constant sequences (e.g., growing regular graphs),
general sequences require careful analysis of matchings in \( G_n \). {To achieve this, we will rely on Vizing’s theorem and its consequence stated in Lemma~\ref{lem:Erdos et al}.}

The DPG process avoids global restructuring and ensures degree invariance,
        but its applicability relies on the non-trivial existence of sufficiently large matchings
        in intermediate graphs.
        It is non-trivial to construct an infinite,
        naturally occurring sequence \( D \) whose every initial segment is graphic.

        In their novel work \cite{EHKMMT},
         Erd\H{o}s \emph{et al.} demonstrated for the first time that
        the sequence of the first \( n \) prime gap graphs constitutes a non-trivial,
        naturally arising infinite structure (for $n$ sufficiently large and for all $n$ if RH assumed).
        Specifically, they established two fundamental properties:
        i) All initial segments of this sequence are graphic;
        ii)  Every realization of these initial segments remains admissible for the DPG-algorithm.

        Our second result makes this explicit.
        \begin{theorem}\label{thm:DPG}
                Let \( n\geq \exp\exp(34.33)\). For any realization \( G_n \) of the prime gap sequence $\mathrm{PD}_n$, the pair \( \big( G_n,\, p_{n+1} - p_n \big) \) is DPG-graphic.
        \end{theorem}

\subsection{Notations.}
In what follows, we use the following notations.
The Chebyshev function is defined by
$$\psi(x)=\sum_{n\leq x}\Lambda(n),$$
where $\Lambda(n)$ is the von-Mangoldt function.
The notation $f=g+O^{*}(h)$ is equivalent to $|f-g|\leq h$.
For $1/2\leq\sigma<1$ and $T>0$, we let $N(\sigma,T)$ be the number
of non-trivial zeros $\rho=\beta+\text{i}t$ of the Riemann zeta function $\zeta(s)$
in the rectangle $\sigma<\beta<1$ and $0<t<T$.

\section{Preliminaries}\label{sec:preliminaries}
\subsection{Some explicit estimates}
Let $H_0\in\mathbb{R}$ be such that the {RH} is known to be true up to $H_0$. That is, all zeroes $\beta+\mathrm{i}\gamma$ of the Riemann zeta-function with $|\gamma|<H_0$ have $\beta=1/2$. Moreover, all of these zeroes are simple. Currently, we may take $H_0=3\times10^{12}$ (see {Platt and Trudgian\cite[Theorem 1]{PT}}).
For $3/5\leq \sigma \leq 1$ and $T> H_0$, we have the following zero-density estimate
(see Kadiri, Lumley and Ng \cite[Section 5, Table 1]{KLN})
\begin{align}\label{ZDE}
N(\sigma,T)
&\leq {2.177} (\log T)^{5-2\sigma} T^{\frac{8}{3}(1-\sigma)} + 5.663\log^2 T\\
&\leq {2.375} (\log T)^{5-2\sigma} T^{\frac{8}{3}(1-\sigma)}.\label{BZD}
\end{align}
{For \eqref{BZD}}, we used the fact that the logarithm contribution
in the first term of \eqref{ZDE} is at least $\log^3 T$
and that for $T>H_0$, \begin{align*}
\frac{5.663}{\log T}<{0.198}.
\end{align*}

By the zero-free region of Ford \cite[Theorem 5]{Ford}, we have that for $|t|>3$, every non-trivial zero $\beta+\mathrm{i}t$ of the Riemann zeta function satisfies 

\begin{align}\label{ZFR}
\beta \leq\, 1-\frac{c_0}{(\log t)^{\frac{2}{3}}(\log\log t)^{\frac{1}{3}}}:=1-\eta(t),
\end{align}
where $c_0 = 1/57.54$.
The constant $c_0$ has recently been improved,
such as in \cite{Bellotti2024,MTYang2024},
with the current best result being $c_0 = 1/51.34$
by Yang \cite[Theorem 1.3]{Yang25}.
We therefore proceed with $c_0 =1/ 51.34$ for our purpose.

Let $N(t)$ denote the number of zeros of $\zeta(s)$ in the critical strip to height $t$.
\begin{lemma}[{Cully-Hugill and Johnston \cite[Lemma 4.2]{CJ23}}]\label{lem:rect-bound}For all $t > 1$,
we have 
\[
N(t+1) - N(t-1) < \log t.
\]
\end{lemma}
%\begin{proof}
%The case \( t > 50 \) follows from Dudek \cite[Lemma 2.6]{Dudek}.
%For \( {\color{red} 1} < t \leq 50 \), the inequality holds by direct verification of zeros of \( \zeta(s) \). 
%{\color{red}See \cite[Lemma 4.2]{CJ23}}. 
%tabulated in  the LMFDB database (\url{https://www.lmfdb.org/zeros/zeta/}).
%\end{proof}

\begin{lemma}[{Delange \cite{Delange}}] \label{lem:Delange_bound}
For any $\sigma>1$ and $t\in \mathbb{R}$ we have
\begin{equation*}
   \left| \frac{\zeta'(\sigma+\mathrm{i}t)}{\zeta(\sigma+\mathrm{i}t)} \right|<\frac{1}{\sigma-1}-\frac{1}{2\sigma^2}.
\end{equation*}
\end{lemma}

\begin{lemma}[{Erd\H{o}s \emph{et al}. \cite[Lemma 2.4]{EHKMMT}}] \label{lem:dudek_bound}
Let $\sigma \le -1$ and $t \in \mathbb{R}$. Assume that either $\sigma \in 1 + 2\mathbb{Z}$ or $|t| \ge 1$. Then
\begin{equation*}
    \left| \frac{\zeta'(\sigma+\mathrm{i}t)}{\zeta(\sigma+\mathrm{i}t)} \right| < 9 + \log|\sigma+\mathrm{i}t|.
\end{equation*}
\end{lemma}

\begin{lemma}[{Dudek \cite[Lemma 2.8]{Dudek}}] \label{lem:help shift contour}
Let $\sigma > -1$ and $t > 50$. Then there exists $t_0 \in (t - 1, t + 1)$ such that for every $\sigma > -1$ we have
\[\left| \frac{\zeta'(\sigma +\mathrm{i}t_0)}{\zeta(\sigma +\mathrm{i}t_0)} \right|< \log^2 t + 20 \log t.\]
\end{lemma}
\begin{remark}This lemma allows us to shift the contour slightly to a line where we have good bounds, if our contour  is somewhat close to a zero.
\end{remark}

Next we establish an explicit formula for the Chebyshev function $\psi(x)$.
\begin{lemma} \label{lem:explicit_formula_real_x}
For any $x\in\mathbb{R}\backslash\mathbb{Z}$ with $x > 10^{18}$ and any $T\in\mathbb{R}$ satisfying $50 < T < x-1$, we have
\begin{equation*}
\psi(x) = x - \sum_{|\Im\rho|<T} \frac{x^\rho}{\rho} +
O^*\left( 4.6 \frac{x \log x\log\log x}{T} \right),
\end{equation*}
where the sum is over the non-trivial zeros of $\zeta(s)$.
\end{lemma}

\begin{proof}
This proof follows the framework laid out in Section 6 of \cite{EHKMMT} with some necessary modifications.
We start with the truncated Perron's formula, as presented in \cite[Eq.\,(19)]{EHKMMT}.
To avoid the potential issue that  $T$ might lie close to the ordinate of a zero,
we choose some $ T_0 \in (T - 1, T + 1)$ according to Lemma \ref{lem:help shift contour}.
For a non-integer $x > 2$ and $c=1+1/\log x$, we have
\begin{equation*} \label{eq:our_perron_psi}
\psi(x) = \frac{1}{2\pi \mathrm{i}} \int_{c-\mathrm{i} T_0}^{{c+\mathrm{i} T_0}} \left(-\frac{\zeta'(s)}{\zeta(s)}\right)\frac{x^s}{s}\mathrm{d}s + O^*\left(\sum_{n=1}^\infty \Lambda(n) \left(\frac{x}{n}\right)^c \min\left(0.501, \frac{1}{\pi T_0 \left|\log\frac{x}{n}\right|}\right)\right).
\end{equation*}
We estimate this  error term by splitting the sum over $n$ into four parts:
$S_1$ for $1 \le |n-x| \le \log x$, $S_2$ for $\log x < |n-x| \le x/5$, $S_3$ for $|n-x| > x/5$, and $S_4$ for $|n-x|<1$.

For $S_1$, we use the second term in the $\min$ function. By \cite[Eq.\,(23)]{EHKMMT}, we have
\begin{equation*}
    S_1\le 1.001 \frac{x \log x}{\pi T_0}
\sum_{\substack{1 \le |x - n| \le \log x \\ \Lambda(n) \ne 0}}
\frac{1}{|x - n|}
%< 0.638 \frac{x\log x}{T}
%\sum_{1 \le k \le \log x}
%\frac{1}{k}
< 0.638\frac{x\log x(\log\log x+0.6)}{T_0}.
\end{equation*}
For $S_2$, by \cite[Eqs.\,(24)--(26)]{EHKMMT}, we have
\begin{equation*}
S_2\le \frac{1}{4 \log \tfrac{5}{4}} \cdot
\frac{x \log x}{\pi T_0}
\sum_{\substack{\log x < |x - n| \le x \\ \Lambda(n) \ne 0}}
\frac{1}{|x - n|}<1.427\frac{x\log x(\log\log x-0.4)}{T_0}.
\end{equation*}
For $S_3$, $|\log(x/n)| > \log(6/5)$ is a constant. {Using Lemma \ref{lem:Delange_bound}, the contribution is bounded by}
\begin{equation*}
    S_3 \le \frac{1}{\pi T_0 \log(6/5)} {\sum_{|n-x|>x/5}}\Lambda(n)\left(\frac{x}{n}\right)^c
    < \frac{ex}{\pi T_0 \log(6/5)}\left|-\frac{\zeta'(c)}{\zeta(c)}\right|<4.746 \frac{x \log x}{T_0}.
\end{equation*}
Finally for $S_4$, there are two $n$'s satisfying $|x - n| < 1$.
We use the first term in the $\min$ function, and by \cite[Eq.\,(20)]{EHKMMT}, the contribution is
\[
S_4\le 0.501 \left( \frac{x}{x - 1} + \frac{x}{x} \right) \log x < 1.003 \log x<1.003\frac{x\log x}{T_0}.
\]
So we get
\begin{equation}\label{eq:psi_before Cauchy}
    \psi(x) = \frac{1}{2\pi \mathrm{i}} \int_{c-\mathrm{i} T_0}^{{c+\mathrm{i} T_0}}
    \left(-\frac{\zeta'(s)}{\zeta(s)}\right)\frac{x^s}{s}\mathrm{d}s +
    O^*\left( {2.065} \frac{x \log x\log\log x}{T_0}+{5.561 }\frac{x\log x}{T_0} \right).
\end{equation}

%We now look to shifting the line of integration so that we might involve the residues of the integrand. In doing so, we incur errors which only slightly increase the above error term. That is, the bulk of the error has already been obtained, and so we can be excused for not pursuing the best possible bounds in the remainder of this section
The integral here is evaluated by the residue theorem, yielding
\begin{align*}
    \frac{1}{2\pi \mathrm{i}}& \left( \int_{c-\mathrm{i}T_0}^{c+\mathrm{i}T_0}  +
    \int_{c+\mathrm{i}T_0}^{-\infty+\mathrm{i}T_0}
    +\int_{-\infty+\mathrm{i}T_0}^{-\infty-\mathrm{i}T_0}
    + \int_{-\infty-\mathrm{i}T_0}^{c-\mathrm{i}T_0} \right)
    \left(-\frac{\zeta'(s)}{\zeta(s)}\right)\frac{x^s}{s}\mathrm{d}s\\
    &= x - \sum_{|\Im\rho|<T_0} \frac{x^\rho}{\rho} - \log(2\pi) - \frac{1}{2}\log(1-x^{-2}).
\end{align*}
By Lemma \ref{lem:dudek_bound}, the contribution of the third integral is
\[
\lim_{U\to\infty}\frac{1}{2\pi}
    \left| \int_{-U-\mathrm{i}T_0}^{-U+\mathrm{i}T_0} -\frac{\zeta'(s)}{\zeta(s)} \frac{x^s}{s} \mathrm{d}s \right|
    \le {\lim_{U\to\infty}\int_{-T_0}^{T_0} \frac{9 + \log \sqrt{U^2 + T_0^2}}{2\pi x^U \sqrt{U^2+t^2}} \mathrm{d}t}
    =0.
\]
By Lemma \ref{lem:dudek_bound} and Lemma \ref{lem:help shift contour}, the contribution of the second integral is
\[
 \frac{1}{2\pi } \left| \int_{c+\mathrm{i}T_0}^{-\infty+\mathrm{i}T_0} \left(-\frac{\zeta'(s)}{\zeta(s)}\right)\frac{x^s}{s}\mathrm{d}s\right|
 \le \frac{{\log^2T+20\log T}}{2\pi T_0}\int_{-\infty}^cx^{\sigma}\mathrm{d}\sigma\le\frac{ex(\log x+20)}{2\pi T_0}.
\]
The fourth integral obeys the same bound. Plugging all these into \eqref{eq:psi_before Cauchy} we have
\begin{equation}
\begin{aligned}\label{eq:psi_explicit error}
    \psi(x)=x&-\sum_{|\Im\rho|<T_0} \frac{x^\rho}{\rho} - \log(2\pi) - \frac{1}{2}\log(1-x^{-2})+O^*\left(\frac{ex(\log x+20)}{\pi T_0} \right)\\
    &+O^*\left( 2.065 \frac{x \log x\log\log x}{T_0}+{ 5.561}\frac{x\log x}{T_0} \right).
\end{aligned}
\end{equation}
By Lemma \ref{lem:rect-bound}, we adjust the $\rho$-sum to $|\Im\rho| < T$ at the cost of an error of
\begin{equation}\label{eq:cost T_T0}
\left|\sum_{|\Im\rho-T|\le 1}\frac{x^{\rho}}{\rho}\right|\le\sum_{|\Im\rho-T|\le 1}\frac{x}{T-1}\le\frac{2x\log T}{T-1}.
\end{equation}
Note that for $x>10^{18},\,T>50$, we have
$$1<0.269\log\log x,\quad 1<0.007\log x\log\log x, \quad\text{and}\ \frac{1}{T-1}\le1.021\frac{1}{T}. $$
Combining \eqref{eq:psi_explicit error} and \eqref{eq:cost T_T0}  we have, for a non-integer $x$,
\begin{equation*}
\psi(x) = x - \sum_{|\Im\rho|<T} \frac{x^\rho}{\rho} + O^*\left(  4.6 \frac{x \log x\log\log x}{T} \right).
\end{equation*}

This completes the proof of this lemma.
\end{proof}
    
\begin{remark}
Note that Lemma \ref{lem:explicit_formula_real_x} gives an explicit version of a result due to Goldston \cite[Eq.\,(4)]{Goldston}; see also another explicit version in Cully-Hugill \cite[Theorem 2]{Cully23} for half odd integers.
\end{remark}

Let $\pi(x)$ be the prime-counting function defined to be the number of primes less than or equal to $x$. 
\begin{lemma}[Dusart\cite{Dusart1998}]\label{lem:Dusart-bound}
        For all  $x \geq 599$, we have
        $$
       \pi(x)\ge \frac{x}{\log x}\left(1+\frac{1}{\log x}\right). %\leq \frac{x}{\log x}\left(1+\frac{1}{\log x}+\frac{2.51}{\log^2x}\right).
        $$
     %   The first inequality holds   and the second one for  $x \geq 355991$.
\end{lemma}

        \subsection{Criteria for sequences being graphic}\label{sec:graph-criteria}
        To address this problem, we draw upon some
knowledge from graph theory. Many results provide necessary and sufficient conditions for a
nonnegative sequence to be graphic.
        One classic and well-known theorem is the
        Erd\H{o}s--Gallai criterion~\cite{EG}:
        Given \( \left(d_i\right)_{i=1}^n \in \mathbb{N}^n \)
        such that \( d_1 \geq d_2 \geq \cdots \geq d_n \geq 0 \)
        and the sum of all terms is even,
        then there exists a (simple undirected) graph with \(n\) vertices
        and degrees \(d_1, \dots, d_n\) if and only if
        \[
        \sum_{i=1}^{k} d_i \leq k(k - 1) + \sum_{i=k+1}^{n} \min(k, d_i)
        \]
        for all \(k \in \{1, \dots, {n-1}\}\).
        %This criteria was also used in \cite{EHKMMT} to establish their results.

        In fact, the results of Zverovich and Zverovich \cite{ZZ92}, and later Tripathi and Vijay \cite{TV}, indicate that it is unnecessary to verify the above inequality for all values of \(1\leq  k\leq n \).
        \begin{theorem}\label{thm:ZZTV}
                Let \( \left(d_i\right)_{i=1}^n \in \mathbb{N}^n \) be a nonincreasing sequence with \( d_1 \geq d_2 \geq \cdots \geq d_n \geq 0 \). Then \( \left(d_i\right)_{i=1}^n\)  is graphic if and only if $\sum_{i=1}^n d_i$ is even and
                \[
                \sum_{i=1}^{k} d_i \leq k(k - 1) + \sum_{i=k+1}^{n} \min(k, d_i)
                \]
                holds for \(k = m\) and for all \( k < m \) such that \( d_k > d_{k+1} \), where \( m = \max \{i : d_i \geq i\} \).
        \end{theorem}
        This result significantly reduces the number of \( k \)-values that need to be checked 
        in the original Erd\H{o}s--Gallai criterion, especially when \( n \) is sufficiently large. 
        This reduction plays an important role in our setting.

\subsection{Application of Vizing's theorem}
A well-known theorem by Vizing~\cite{Viz64} states that a simple graph with maximal degree $\Delta$ 
admits a proper edge coloring with $\Delta + 1$ colors. This leads to the following lemma, 
which establishes a condition for the existence of a large matching.

        \begin{lemma}[{Erd\H{o}s \emph{et al}.\cite[Lemma 2.1]{EHKMMT}}]\label{lem:Erdos et al}
                Let $G$ be a simple graph on $n$ vertices with degrees $d_1, \dots, d_n$. Let $\delta \geq 1$ be an integer, and let $d \geq 2$ be an even integer satisfying
                \[
                \delta d \leq \sum_{d_i < \delta} d_i - \sum_{d_i \geq \delta} d_i.
                \]
                Then $G$ has a matching of size $d/2$.
        \end{lemma}

        \section{First moment of large prime gaps}\label{sec:first moment}
        Let \( N = 4\delta x \) throughout this section. For simplicity, we define the integral
        \[\mathcal{I}=\mathcal{I}(x,T):=\int_x^{2x} \left| \sum_{|\Im \rho| < T} y^{\rho} C(\rho) \right|^2 \mathrm{d}y\]
        where
        $$C(\rho)=\frac{1-(1+\delta)^{\rho}}{\rho}$$
        with
        \( 0 < \delta \leq 1 \). Following essentially the method of Heath-Brown~\cite{HB},
        we will bound the first moment of large prime gaps by the integral $\mathcal{I}$.

\begin{lemma}\label{pro:1st-moment}
    Let $x\geq \exp{10^3}$ and $N>8\log^2x.$ We have
        \begin{align*}
        \sum_{\substack{x\leq p_\ell\leq 2x\\p_{\ell+1}-p_\ell> N}}
        \left(p_{\ell+1}-p_\ell\right)\leq \max_{\substack{x\leq p_\ell\leq 2x}}
        \left(p_{\ell+1}-p_\ell\right)+ 8.1\frac{1}{\delta^2x^2}\mathcal{I}(x,2\delta^{-1}\log^2x).
        \end{align*}
                %where the pairs $\kappa$ and $x_\kappa$ are given in \textcolor{red}{TABLE}.
                %where $N^*$=\max \{p_{l+1}-p_l \mid  x\leq p_l\leq 2x, p_{l+1}-p_l\geq N\}.$
                %denote the largest prime gap $p_{l+1}-p_l$ occurring in the above .
                %       where \( G_x \) denotes the largest prime gap \( p_{l+1}-p_l \) occurring among primes \( p_l \) in the interval \( [x, 2x] \) and \( p_{l+1}-p_l\geq N\).
        \end{lemma}

        \begin{proof}
                Assume that the prime $p_\ell\in[x,2x]$ satisfies ${p_{\ell+1}}-p_\ell\geqslant N.$ There is at most one $p_\ell$ such that
                $(p_{\ell}+p_{\ell+1})/2>2x$, so we consider its contribution separately and assume that $(p_{\ell}+p_{\ell+1})/2\leqslant2x$. Then, for any
                $$y\in(p_\ell,(p_\ell+p_{\ell+1})/2)\subset(x,2x),$$
                the interval
                $$[y,y+\delta y]\subset[y,y+N/2]\subset(p_\ell,p_{\ell+1})$$
                is free of primes.
                The difference \(\psi(y+\delta y) - \psi(y)\) therefore only includes contribution
                from higher prime powers \(p^k\) (\(k\geq 2\)), which gives
                \begin{equation*}
                        \begin{aligned}
                                \sum_{\substack{p^k \in (y, y+\delta y] \\ k\geq 2}} \log p
                                &\leq  \sum_{2 \leq k
                                        \leq \frac{\log(y+\delta y)}{\log 2}}\,\sum_{\substack{y^{1/k}< p \leq  (y+\delta y)^{1/k}}} \log p\\
                                &<\sum_{k\ge 2}\frac{1}{k}\log(y+\delta y)\left(\pi\big((y+\delta y)^{1/k}\big)-\pi\big( y^{1/k}\big)\right) \\
                                &<\log(y+\delta y)\sum_{k\ge 2}\frac{1}{k}y^{1/k}\left((1+\delta)^{1/k}-1\right) \\
                                &<  \left(\frac{\pi^2}{6}-1\right)\delta\sqrt{y}\log(2 y)
                                < 0.001\delta y,
                        \end{aligned}
                \end{equation*}
                for all $y\geq x\geq1.6\times 10^8$.
                Now we have $|\psi(y+\delta y)-\psi(y)-\delta y|>0.999\delta y$ holds on
                $y\in(p_\ell,(p_\ell+p_{\ell+1})/2).$ Squaring and integrating, we obtain
                $$\int_{p_{\ell}}^{(p_{\ell}+p_{\ell+1})/2}
                |\psi(y+\delta y)-\psi(y)-\delta y|^{2}\mathrm{d}y
                >0.499\delta^{2}x^{2}(p_{\ell+1}-p_{\ell}).$$
                Note that the set of points where $y+\delta y$ or $y$ takes an integer value
                is countable and thus of Lebesgue measure zero, making its contribution to the above integral zero.
                It follows from {Lemma~\ref{lem:explicit_formula_real_x}} that
                \begin{align*}
                        |\psi(y+\delta y)-\psi(y)-\delta y |& = \left|\sum_{|\Im\rho|<T}\frac{y^{\rho}-(y+\delta y)^{\rho}}{\rho} +E(y+\delta y,T)-E(y,T)\right| \\
                        & \leq \left|\sum_{|\Im\rho|<T}y^{\rho}C(\rho) \right|+E(T,x).
                \end{align*}
                With the choice of $T=2\delta^{-1}\log^2x<x$, the error term $E(T,x)$ satisfies
                \begin{align*}
                    E(T,x) & \le  {4.6}\frac{4x\log(4x)\log\log(4x)+2x\log (2x)\log\log(2x)}{T}\\
                        &\le \frac{x}{T}\log^2x=\frac{1}{2}\delta x,
                \end{align*}
                for all {$x\ge 3.9\times 10^{59}$}. On using the simple inequality $(A+B)^2\leq 2A^2+2B^2$, we have
                $$0.499\delta^{2}x^{2}(p_{\ell+1}-p_{\ell})< 2\int_{p_{\ell}}^{(p_{\ell}+p_{\ell+1})/2}\left|\sum_{|\Im\rho|<T}y^{\rho}C(\rho) \right|^2 \mathrm{d}y+\frac{1}{4}\delta^2 x^2 (p_{\ell+1}-p_{\ell}). $$
                Summing over all such primes $p_\ell$,
                and bounding the possible single exceptional $p_\ell$ by taking maximum,
                we obtain the desired result.
        \end{proof}
        \begin{remark}
                Lemma \ref{pro:1st-moment} is also established in \cite[Section 5]{EHKMMT}.
                Here, we present a reformulation and a minor modification to it.
                Specifically, the conclusion of Lemma \ref{eq:our_perron_psi} ensures validity for \( T < x \), which allows us to take a smaller imaginary part.% thereby leading to an improved result.
                %{\color{blue}When $\kappa=10^{{\color{red}18}}$, one can take $c_\kappa=1/0.497$}.
        \end{remark}

        The main task now is to give the non-trivial estimate of Integral $\mathcal{I}$ explicitly.
        This can be done by following the same idea as
        in {Saffari and Vaughan}~\cite[Lemma 5]{saffari1977}.
        %and an explicit formula for the Chebyshev function {\color{red} in Lemma \ref{lem:explicit-formula}}.

        \begin{proposition}[Bounds for Integral $\mathcal{I}$]\label{pro:Selberg-bound}
                Suppose that there are positive constants \(c_1, A\) (with \( {2\leq A \leq 4 }\)) such that
                \begin{equation}\label{eqn:ZDE}
                        N(\sigma, t) \leq c_1 t^{A(1-\sigma)} (\log t)^{5-2\sigma}, \quad \left(t \geq 2, \,\,\sigma\geq 1-1/A\right).
                \end{equation}
                Suppose further that the Riemann zeta function has no zeros $\rho=\beta+\mathrm{i}\gamma$ in the region
                \begin{equation}\label{eqn:ZFR}
                        \beta\geq 1-\eta({\gamma}),\quad {\gamma}\geq 2,
                \end{equation}
                with $\eta ({\gamma})<1/A$.
                Then, whenever \( x \geq \exp{10^3}\), \(0<\alpha<2/A \)  and
                \(10< T\leq x^{2/A-\alpha} /\log x\), we have
                \[
                \mathcal{I}(x,T)
                \leq    C\,\delta^2     x^{3-A\alpha \eta(T)}\log^4 x,
                \]
                where  the constant\[C=C(A,\alpha,c_1)=4\times1836\,\frac{c_1}{A\alpha} \left(\frac{2}{A}-\alpha\right)^{4}.\]

        \end{proposition}
        \begin{proof}
                We can first bound this integral as follows
                \begin{align*}
                        \begin{split}
                                \mathcal{I}
                                &\leq  \int_1^{2} \int_{\frac{vx}{2}}^{2vx} \left| \sum_{|\Im \rho| \leq T} y^{\rho} C(\rho) \right|^2 \mathrm{d}y \, \mathrm{d}v\\
                                &\leq \int_1^2 \sum_{|\Im \rho| \leq T} \sum_{|\Im \rho'| \leq T} C(\rho) \overline{C(\rho')} \frac{2^{1+\rho + \overline{\rho'}} - 2^{-1 - \rho - \overline{\rho'}}}{1 + \rho + \overline{\rho'}} x^{1+\rho + \overline{\rho'}}v^{1+\rho + \overline{\rho'}} \mathrm{d} v\\
                                &\leq  x \sum_{|\Im \rho| \leq T} \sum_{|\Im \rho'| \leq T} \left|x^{\rho + \overline{\rho'}} C(\rho) \overline{C(\rho')}\right|\cdot  \left| \frac{2^{1+\rho + \overline{\rho'}}-2^{-1 - \rho - \overline{\rho'}}}{1 + \rho + \overline{\rho'}}\right|\cdot \left|\frac{2^{2+\rho + \overline{\rho'}}-1}{2+\rho + \overline{\rho'}} \right|.
                        \end{split}
                \end{align*}
                By writing $\rho=\beta+\mathrm{i}\gamma, \rho'=\beta'+\mathrm{i}\gamma'$, and using the elementary inequality $2|z_1 z_2| \leq |z_1|^2 + |z_2|^2$, the above is
                \[
                \leq \frac{1}{2}x \sum_{\substack{\rho=\beta+\mathrm{i}\gamma\\|\gamma| \leq T}} \sum_{\substack{\rho'=\beta'+\mathrm{i}\gamma'\\|\gamma'| \leq T}}\left(x^{2\beta}|C(\rho)|^2+x^{2\beta'}|C(\rho')|^2\right)\left(\frac{(2^3+1)(2^4+1)}{1+|\gamma-\gamma'|^2} \right),
                \]
                which gives
                \begin{align*}
                        \mathcal{I}&\le 153x\sum_{\substack{\rho=\beta+\mathrm{i}\gamma\\|\gamma| \leq T}}x^{2\beta}|C(\rho)|^2 \sum_{\substack{\rho'=\beta'+\mathrm{i}\gamma'\\|\gamma'| \leq T}}\frac{1}{1+|\gamma-\gamma'|^2 }\\
                        & \le 4\times153x\sum_{\substack{\rho=\beta+\mathrm{i}\gamma\\ \beta \geq \frac{1}{2}, 10<\gamma\leq T}} x^{ 2 \beta}|C(\rho)|^2\sum_{\substack{\rho'=\beta'+\mathrm{i}\gamma'\\ |\gamma'|\leq T}}\frac{1}{1+|\gamma-\gamma'|^2 }
                \end{align*}
                by the symmetry of the double summation and the symmetric distribution of zeros. We divide the summation over $\gamma'$
                by their distance from $\gamma$, and apply Lemma~\ref{lem:rect-bound} to get that the inner sum of above is
                \begin{align*}
                        & \leq \sum_{k\ge0}\,\sum_{k\le|\gamma'-\gamma|\le k+1}\frac{1}{1+|\gamma'-\gamma|^2}\le \sum_{k\ge 0}\frac{1}{1+k^2}\sum_{k\le |\gamma'-\gamma|\le k+1}1\\
                        &\leq \sum_{k\ge0} \frac{\log (\gamma+k)}{1+k^{2}}
                        \leq \sum_{k=0}^{\lfloor\gamma\rfloor} \frac{\log (2\gamma)}{1+k^{2}}+\sum_{k>\gamma}
                        \frac{\log (2 k)}{k^{2}}\le \left(\frac{1}{2}+\frac{\pi}{2}\coth(\pi)\right)\log(2\gamma)+\frac{\log(2\gamma)+1}{\gamma} \\
                        &\leq 3\log\gamma
                \end{align*}
                for all $\gamma>10$.
                %The constant in the final inequality must be chosen to ensure that the bound derived in this step holds uniformly for all $\gamma\geq 2$.
                Using the integral representation
                \[
                C(\rho) = \int_{1}^{1+\delta} x^{\rho-1} \mathrm{d}x,
                \]
                we apply the triangle inequality for complex-valued integrals to get
                \[
                |C(\rho)| \leq \int_{1}^{1+\delta} |x^{\rho-1}| \mathrm{d}x = \int_{1}^{1+\delta} x^{\beta-1} \mathrm{d}x \leq \int_{1}^{1+\delta} 1 \, \mathrm{d}x = \delta.
                \]
        Thus we have
                \begin{align}\label{eqn:SI-bound-2parts}
                        \begin{split}
                                %& \mathcal{I}\leq 4\times 153 x\sum_{\substack{\rho=\beta+i\gamma\\ \beta \geq \frac{1}{2}, 0<\gamma\leq T}} x^{ 2 \beta} \min(\delta^2, \frac{9}{\gamma^2})\sum_{\substack{\rho'=\beta'+i\gamma'\\|\gamma'| \leq x}}\frac{1}{1+|\gamma'-\gamma|^2}  \\
                                &\mathcal{I}\leq  1836 x\delta^2\log T \sum_{\substack{\rho=\beta+\mathrm{i}\gamma\\ \beta \geq \frac{1}{2}, 10<\gamma\leq T}} x^{2\beta}.
        \end{split}
                \end{align}
        Let$$S=\sum_{\substack{\rho=\beta+\mathrm{i}\gamma\\ \beta \geq \frac{1}{2}, 10<\gamma\leq T}} x^{2\beta}=S_1+S_2,$$
        with
        $$S_1:=\sum_{\substack{\rho=\beta+\mathrm{i}\gamma\\ \beta \geq \frac{1}{2}, 10<\gamma\leq T}} \left(x^{2\beta}-x\right), \quad\text{and}\quad  S_2:=x\sum_{\substack{\rho=\beta+\mathrm{i}\gamma\\ \beta \geq \frac{1}{2}, 10<\gamma\leq T}}1 .$$
                By Lemma~\ref{lem:rect-bound}, we have
                \begin{equation}\label{eqn:S12}
                        \begin{aligned}
                                S_{2}&\le xN(T)\le xT\log T.
                        \end{aligned}
                \end{equation}
                By the zero-free region \eqref{eqn:ZFR},  we have
                \begin{align*}
                        S_{1}&=\sum_{\beta \geq \frac{1}{2},\,10
                                <\gamma \leq T}\int_{1/2}^{\beta}2x^{2\sigma}\log x\ \mathrm{d}\sigma\\
                        &\leq \int_{1/2}^{1-\eta(T)}2x^{2\sigma}(\log x)\,N(\sigma,T) \,\mathrm{d}\sigma.
                        %&=2(\log x) \log(3\delta^{-1})\left(\int_{1/2}^{1-1/A}+\int_{1-1/A}^{1-\eta(3\delta^{-1})}  \right)x^{2\sigma}N(\sigma,3\delta^{-1}) \, {d}\sigma .
                \end{align*}
                We split the integral above into two parts, one over the interval  $1/2 \leq \sigma \leq 1-1/A$, and another over  $1-1/A \leq \sigma \leq   1-\eta(T) $. By Lemma~\ref{lem:rect-bound}, the contribution of first interval is
                \begin{align}\label{eqn:S11-1}
                %       &\le T\log T\int_{1/2}^{1-1/A}2x^{2\sigma}\log x\, d\sigma \nonumber \\
                        &\le T\log T (x^{2-\frac{2}{A}}-x).
                \end{align}
                And the contribution over $[1-1/A,1-\eta(T)]$ is estimated by the zero-density estimate~\eqref{eqn:ZDE}, which gives
                \begin{align}\label{eqn:S11-2}
                        &\le 2c_1\log x \int_{1-1/A}^{1-\eta(T)}x^{2\sigma}T^{A(1-\sigma)}(\log T)^{5-2\sigma}\,\mathrm{d}\sigma \nonumber \\
                        &\le 2c_1x^2\log x \log^3 T\,\frac{K^{-\eta(T)}}{\log K},
                \end{align}
                where
                $$K=\frac{x^2 }{T^A\log^2 T}.$$
                In the end, since it is assumed that \( K \geq x^{A \alpha} > 1 \) and $\eta(T)<1/A$, by combining \eqref{eqn:S12}, \eqref{eqn:S11-1} and \eqref{eqn:S11-2} we obtain
                \begin{equation}\label{eqn:SlogT}
                        \begin{aligned}
                                S\cdot\log T&\leq \left( \frac{2}{A}-\alpha \right)^2 x^{2-\alpha}\log^2 x+
                                \frac{2c_1}{A\alpha} \left(\frac{2}{A}-\alpha\right)^{4}
                                x^{2-A\alpha \eta(T)}\log^4 x.
                        \end{aligned}
                \end{equation}
                By inserting \eqref{eqn:SlogT} into \eqref{eqn:SI-bound-2parts}, we complete the proof.
    \end{proof}

        With our choice \(T=2\delta^{-1}\log^2x\), we have\[
        \mathcal{I}(x,2\delta^{-1}\log^2x)\le C\delta^2x^{3-A\alpha\eta(2\delta^{-1}\log^2x)}\log^4x.
        \]
        This is valid for $N\ge 8x^{1+\alpha-2/A}\log^3x.$
        Applying the best known explicit results on the zero-density estimate \eqref{BZD}
        and the zero-free region \eqref{ZFR}, we find that
        $A=8/3, \,c_1={2.375}$,
        and
        \[\eta(T)=\eta\left(\frac{8x\log^2x}{N}\right)
        =\frac{1}51.34\left(\log \frac{8x\log^2x}{N}\right)^{-2/3}\left(\log\log \frac{8x\log^2x}{N}\right)^{-1/3}.\]
        Proposition~\ref{pro:Selberg-bound}, together with Lemma~\ref{pro:1st-moment} then gives
        \begin{align}
                \sum_{\substack{x\leq p_\ell\leq 2x\\p_{\ell+1}-p_\ell> N}}\left(p_{\ell+1}-p_\ell\right)
                &\leq \max_{\substack{x\leq p_\ell\leq 2x}}
                \left(p_{\ell+1}-p_\ell\right)
                +8.1C x^{1-\frac{8}{3}\alpha\eta(8N^{-1}x\log^2x)}\log^4x.\label{daydic gap summ}
        \end{align}
        Hence, we establish the following theorem,
        which will be primarily used to prove Theorem~\ref{MainThm} and Theorem \ref{thm:DPG}.
        \begin{theorem}\label{thm:main-thm}
                Let $x\geq \exp(4\times10^3),\, N\geq 8x^{1/4+\alpha}\log^3x$ with $0<\alpha< 3/4$.
                Then
                \begin{equation*}
                        S_N(x):=\sum_{\substack{ p_\ell\leq x\\p_{\ell+1}-p_\ell> N}}\left(p_{\ell+1}-p_\ell\right)\leq M(x)\log x +8.7C\,x^{1-\frac{8}{3}\alpha\eta\left(4N^{-1}x\log^2x\right)}\log^4 x,
                \end{equation*}
                where $M(x)$ is the largest prime gap up to $x$.
        \end{theorem}
        \begin{remark}
                This theorem can be compared with Theorem 1.8 in \cite{EHKMMT},
                where the latter assumes {RH}.
                Informally, our result serves as its unconditional counterpart,
                albeit with a weaker exponent in the upper bound.
                This compromise requires that the subsequent main
                {Theorem \ref{MainThm}} is vaild only for sufficiently large \(n\).
        \end{remark}
        \begin{proof}
        First note that
             $$\sum_{\substack{y/2< p_\ell\leq y\\p_{\ell+1}-p_\ell> N}}(p_{\ell+1}-p_\ell)=0,$$ if $N\ge y$.
                Let $k$ be the integer such that $2^{-k}x<N\leq 2^{-k+1}x$,
                then we have $S_N(2^{-k}x)=0$
                and $$2^{-k}x\geq 2^{-1}N\geq 4x^{\frac{1}{4}+\alpha}\log^3x \geq\exp{10^3}.$$
                Let  $\eta(t)$ be as in \eqref{ZFR} for $t> 3$. Splitting into dyadic intervals and applying \eqref{daydic gap summ} we have
        \begin{align}\label{boundS}
        S_{N}(x)&=\sum_{\substack{\frac{x}{2}\leq p_\ell\leq x\\p_{\ell+1}-p_\ell> N}}(p_{\ell+1}-p_\ell)+ \cdots+\sum_{\substack{\frac{x}{2^{k}}\leq p_\ell\leq \frac{x}{2^{k-1}}\\p_{\ell+1}-p_\ell> N}}(p_{\ell+1}-p_\ell)
        \nonumber\\
        &\leq   M(x)\log x+8.1C\sum_{j=1}^{k} \left(\log \frac{x}{2^j} \right)^{4}\left(\frac{x}{2^j}\right)^{1-\frac{8}{3}\alpha\eta \left(8N^{-1}\frac{x}{2^j}\log^2(\frac{x}{2^j})\right)}.
        \end{align}
        Since $\eta(t)$ is decreasing,  we obtain
        \[
        \eta\left(\frac{8x\log^2(\frac{x}{2^j})}{2^j N}\right)\ge \eta\left(\frac{4x\log^2x}{N} \right),
                \]
        and with $2^{-k}x\geq N/2$ to get
        \begin{align*}
        \eta\left(\frac{8x\log^2(\frac{x}{2^j})}{2^j N}\right)
        &\leq \eta\left(\frac{8x\log^2(\frac{x}{2^k})}{2^k N}\right)\le \eta\left(4\log^2\frac{N}{2} \right)\le \eta(4)< \frac{1}{40}.
        \end{align*}
        Then the summation {on the right-hand side of \eqref{boundS}} is\begin{align*} &\leq x^{1-\frac{8}{3}\alpha\eta\left(4N^{-1}x\log^2x\right)}\log^4 x\,\sum_{j\ge1}\left(\frac{1}{2^j}\right)^{1-\frac{8}{3}\times\frac{3}{4}\times\frac{1}{40}}\\
                        &\leq 1.074\,x^{1-\frac{8}{3}\alpha\eta\left(4N^{-1}x\log^2x\right)}\log^4 x,
                \end{align*}
                where the infinite sum over $j$ is convergent to $1.0731\dots$\,. This completes the proof.
        \end{proof}

        \section{{Proof of Theorem \ref{MainThm}}}\label{sec:proof of main thm}
%       In this section we rearrange the first \( n \) prime gaps sequence \(\left(p_{\ell} - p_{\ell-1}\right)_{\ell=1}^{n}\)
%       into a non-increasing sequence \(\left(g_i\right)\)
%       such that \[g_{n} \ \geq \cdots \geq g_2\geq g_1 .\]
%       Note that \(g_n =M(p_n)\) is the maximum prime gap among the first \( n \) prime gaps,
%       and  $g_4=g_3=2, g_{2}=g_{1}=1$, etc.  Moreover,

        Thanks to  Theorem \ref{thm:ZZTV}, to prove Theorem \ref{MainThm},
        it suffices to check the Erd\H{o}s--Gallai inequalities where the sequence jumps. We define
        %\begin{align*}
%                S_N(x) := \sum_{\substack{p_{\ell}\le x\\p_{\ell+1}-p_{\ell}\geq N }}(p_{\ell+1}-p_{\ell})
%        \end{align*}
%and
        $$k_N(x):=|\{\ell\in\mathbb{N}:\,p_{\ell}\le x, \,p_{\ell+1}-p_{\ell}\geq N\}|,$$
        and write $k_N=k_N(p_n) $ for simplicity.

%       Write \(g_n=g_{\max}(p_n) \) as the maximum prime gap among the first $n$ prime gaps.

        Thus we are required to establish %(see \cite{EHKMMT})
        \begin{align*}
                \sum_{\substack{1\leq \ell\leq n\\p_{\ell+1}-p_{\ell}> N}}(p_{\ell+1}-p_{\ell}) \leq k_N(k_N-1)+n-k_N
        \end{align*}
        for all $ N\geq 2$, which is equivalent to establishing
        \begin{align*}
                S_N(p_n)\leq (k_N-1)^2+n-1
        \end{align*}
for all $1\le k_N\le n$.

        \subsection{Case 1: $k_N\geq \sqrt{p_n}+1$}
        This is the trivial case, since
        \begin{align*}
                S_N(p_n)\leq  p_n\leq (k_N-1)^2.
        \end{align*}

        \subsection{Case 2: $k_N\leq \sqrt{p_n}$}
        In this case, we choose $\alpha=0.249$, and let
        $$N'= \frac{p_n^{1/2}}{3\log p_n}.$$
   It is easy to see that $N'\ge 8p_n^{1/4+\alpha}\log^3p_n$    for $n\geq \exp\exp(10.8)$.
         Then we have by Theorem \ref{thm:main-thm} that
        \begin{align}
                S_N(p_n) &=\sum_{\substack{1\leq \ell\leq n\\N\leq p_{\ell+1}-p_{\ell}< N'}} (p_{\ell+1}-p_{\ell}) +S_{N'}(p_n)\nonumber \\
                &\leq  N' k_N +M(p_n)\log p_n +8.7C\,p_n^{1-\frac{8}{3}\alpha\eta\left(\frac{4p_n\log^2p_n}{N'}\right)}\log^4 p_n.\label{eqn:case2}
        \end{align}
  By the lower bound on $\pi(p_n)$ (Lemma~\ref{lem:Dusart-bound}), we have
        \[
\frac{p_n}{\log p_n}<\frac{p_n}{\log p_n}\bigg(1+\frac{1}{\log p_n}\bigg)-1\leq n-1.
        \]
        So the first term in \eqref{eqn:case2} is
        \begin{equation}\label{eq:S_N first term}
         N'k_N  \leq \frac{p_n}{3\log p_n}< \frac{n-1}{3}.
        \end{equation} 

        Using the fact that there is at least one prime
        between {\(n^{90}\) and \((n+1)^{90}\)}
        for all \(n\geq 1\) (see
        { {Cully-Hugill and Johnston \cite[Theorem 1.4]{CJ25}}}), we crudely bound
        \[
         {M(p_n)
         \leq 180\,\left(p_n^{1/90}+2\right)^{89}
         \leq \frac{p_n}{3\log^2p_n}}
        \]
        for \(n \geq \exp\exp(10)\).
        Thus the second term in \eqref{eqn:case2} is
     \begin{equation}\label{eq:S_N sec term}
            M(p_n)\log p_n \leq \frac{p_n}{3\log^2p_n}\log p_n= \frac{p_n}{3\log p_n}< \frac{n-1}{3}.
        \end{equation}

        For the third term in \eqref{eqn:case2}, we write $t=\log\log p_n$.
        Numerical computations show that the following results
        \[
        \log(3\times 8.7C)+5t<\frac{8\times 0.249}{3\times 51.34}
        \cdot\frac{e^t}{L_1^{2/3}\log^{1/3} L_1 }
        \]
        hold for all {$t\geq 30.32$}, where \[
        L_1=L_1(t)=\frac{1}{2}e^t+3t+\log {12}, \]
        and
        \[
        C=4\times{1836}\,\frac{c_1}{A\alpha} \left(\frac{2}{A}-\alpha\right)^{4} {\approx1654.9\le 1655},
        \]
 with $\alpha=0.249$, $A=8/3$, $c_1={2.375}$ and $c_0=1/51.34$. This implies that
{\begin{equation}\label{eq:S_N 3rd term}
8.7C\,p_n^{1-\frac{8}{3}\alpha\eta\left(\frac{4p_n\log^2p_n}{N'}\right)}\log^4 p_n<\frac{p_n}{3\log p_n}<\frac{n-1}{3}
 \end{equation}}
 for \(n \geq \exp\exp(30.32)\). Therefore by \eqref{eq:S_N first term}--\eqref{eq:S_N 3rd term}, it is immediate that \[S_N(p_n)<3\cdot\frac{n-1}{3}=n-1.\]
        Combining these two cases, we have thus completed the proof of Theorem \ref{MainThm}.
%       \begin{remark}
%               The cut $N'=8p_n^{a}\log^3p_n$ at $a=1/2-\xi$ is purely arbitrary and it might be better (numerically) to do it closer to $1/2$         (and possibly using the second part of the Lemma \ref{MainLemma}). Actually we may even cut at $N'=p_n^{1/2}/(\log p_n)^A$ for some suitable $A>0$.             However, in our proof, the tightest upper bound for the third term in \eqref{eqn:case2} occurs when \( N'=p_n^{1/2} \). Numerical calculations show that this requires \( x \geq \exp(\exp(29.89)) \). Although alternative cut might offer minor optimizations, their practical impact is insignificant. We therefore do not pursue this modification.         \end{remark}

        %\section{Numerical things}\label{sec:NT}
%       Changing the constant in front from $10^7$ to $100$ changed almost nothing.

%       Use Wolframalpha. Input:
%       \begin{verbatim}
%               ln(10^7)+(3/8)*(1/(1/4 -0.001))*(53.989)*4.5*(1-(0.5-0.001))^(2/3)t
%               <(e^t/(ln((1-(0.5-0.001))e^t)))^(1/3)
%       \end{verbatim}
%       Solved:
%       \[t>29.8997.\]
%       Changing the constant from $10^8$ to $10$ gives $$t>29.8928.$$

        \section{About the DPG-process}\label{sec:DPG-process}

        We have proved that the sequence of the first \(n\) prime gaps
        (with \(p_0 = 1\)) is graphic for $n\geq \exp\exp(30.49)$. Let $G_n$ be the realization.
        By Erd\H{o}s \emph{et al}.\cite[Theorem 2.5]{EMTZ24}, 
        to prove that $(p_\ell-p_{\ell-1})_{\ell=1}^n$ is $\textrm{DPG}$-graphic, 
        it suffices to show that $G_n$ has $(p_{n+1}-p_n)/2$ independent edges. 
        Applying Lemma \ref{lem:Erdos et al}, we are required to exhibit an integer $N\geq 2$ satisfying
        \begin{align*}
                N(p_{n+1}-p_{n})+2\sum_{\substack{1\leq \ell\leq n\\ p_\ell-p_{\ell-1}> N}}
                \left(p_\ell-p_{\ell-1}\right)<p_n,
        \end{align*}
        for all $n\geq n_0$ (This is \cite[Eq.\,(11)]{EHKMMT}).
        In \cite{Dudek}, the author shows that for $x>\exp\exp(33.3)$
        one has the bound $p_{n+1}-p_n<3x^{2/3}+3x^{1/3}+2$ when $x<p_n<2x$.
        Choose $\alpha=1/13$ such that
        $$N=\frac{1}{12} p_n^{1/3}\ge 8p_n^{1/4+\alpha}\log^3p_n$$
        for $n\geq \exp\exp(8.5)$, then we have
        $$N(p_{n+1}-p_n)\leq  \frac{3.1}{12}p_n^{1/3+2/3}<\frac{1}{3}p_n.$$
        By Theorem \ref{thm:main-thm},
        we also need the inequality
        $$M(p_n)\log p_n+8.7Cp_n^{1-\frac{8}{3}\alpha\eta(48p_n^{2/3}\log^2p_n)}\log^4
        p_n<\frac{1}{3}p_n .$$
By \eqref{eq:S_N sec term}, it suffices to show that
\[
8.7Cp_n^{1-\frac{8}{3}\alpha\eta(48p_n^{2/3}\log^2p_n)}\log^4
        p_n<\frac{1}{4}p_n<\frac{1}{3}p_n\bigg(1-\frac{1}{\log p_n}\bigg),
\]
for all  \(n \geq \exp\exp(10)\).      This can be reduced  to
    \[ \log (4\times 8.7C)+4t
        <\frac{8}{3\times13\times 51.34}\times\frac{e^t}{L_2^{2/3}\log^{1/3}L_2 },
        \]
 where \[
        L_2=L_2(t)=\frac{2}{3}e^t+2t+\log 48, \]
        and
        \[
        C=4\times{1836}\,\frac{c_1}{A\alpha} \left(\frac{2}{A}-\alpha\right)^{4} {\approx 17451.4\le 17452},
        \]
        by writing $t=\log\log p_n$. By numerical verification, this is true for all $t\geq 34.33$. Therefore, we conclude the proof of Theorem \ref{thm:DPG}.

        \bigskip
        \noindent{\bf Acknowledgements.}
                The authors would like to thank Prof. Gergely Harcos for suggesting the problem and his valuable suggestions on this topic. We also thank Daniel R. Johnston for pointing out the references \cite{Cully23,CJ25}. K. Aggarwal thanks the {R\'enyi} Institute for funding and support where most of the work was done. R. Frot also thanks the {R\'enyi} and the {Erdős} Center for funding and support, which were provided in part by the R\'enyi Int\'ezet Lend\"{u}let Automorphic Research Group.  H. Gou and H. Wang are grateful to Prof. Jianya Liu for his help and encouragement.  H. Gou thanks financial support from the China Scholarship Council and the excellent research environment provided by Universit\'e de Montr\'eal.
                H. Wang also gratefully acknowledges the R\'enyi Institute and Shandong University for providing an excellent research environment, which enabled the preparation of a portion of this work, as well as the China Scholarship Council for supporting her studies there. The authors are very grateful to the referees for their valuable suggestions.

        \bibliographystyle{amsplain}
        
\end{document}